\documentclass[11pt]{article}
\usepackage{psfrag}
\usepackage{amssymb,bbm}
\usepackage{amsmath}
\usepackage{amsfonts}
\usepackage{pstricks}
\usepackage{graphicx}
\usepackage[latin5]{inputenc}
\usepackage{yfonts}
\usepackage{setspace}
\newcounter{theorem}
\newtheorem{theorem}{Theorem}

\newtheorem{lemma}{Lemma}

\newtheorem{example}{Example}

\newtheorem{remark}{Remark}
\newtheorem{definition}{Definition}

\newenvironment{proof}[1][Proof]{\textbf{#1.} }{\rule{0.5em}{0.5em}}

\title{Graph-Directed Sprays and Their Tube Volumes via Functional Equations \date{}}
\author{Derya ÇELİK  \footnote{Anadolu
University, Science Faculty, Department of Mathematics, 26470,
Eskişehir, Turkey, \newline e-mails: deryacelik@anadolu.edu.tr,
 skocak@anadolu.edu.tr, yunuso@anadolu.edu.tr,
aeureyen@anadolu.edu.tr}  \and Şahin KOÇAK $^*$ \thanks{Corresponding Author.}   \and Yunus ÖZDEMİR$^*$ \and A. Ersin
ÜREYEN $^*$}

\begin{document}

\maketitle \doublespacing \pagestyle{myheadings}

\begin{abstract}
 The notion of sprays introduced by Lapidus and his co-workers has proved useful in the context of fractal tube formulas. In the present note, we propose a more general concept of sprays, where we allow several generators to make them more convenient for applications to graph-directed fractals. Using a simple functional equation satisfied by the volume of the inner $\varepsilon$-neighborhood of such a generalized spray, we establish a tube formula for them.
\end{abstract}



\section{Introduction}

Tube volumes for fractals has been an interesting topic of research in the last decade. \mbox{M. Lapidus} and his coworkers established several tube formulas as series in terms of residues of some associated zeta functions \cite{LaFraBook}. Recently, we have proposed a simple alternative approach to tube formulas for self-similar sprays via functional equations (\cite{DKOUJofG}). The so-called sprays introduced by Lapidus and Pomerance in \cite{LaPoCounter} (to be defined below) are a convenient notion in dealing with fractal tube volumes.

Beyond the self-similar fractals, the next family of interest is the class of graph-directed fractals. For tube volumes of graph-directed fractals there are formulas (see \cite{DDKUFractals}), which are analogous to the tube formulas of Lapidus and Pearse (\cite{LaPeActa}). So far as we know, the notion of spray has not been extended to the graph-directed setting. In the present note, we want to propose the notion of graph-directed sprays and establish a tube formula for them via functional equations.

We now first recall the notion of spray in Euclidean space. Let $G\subseteq \mathbb{R}^n$ be a bounded open set. A spray generated by $G$ is a collection $\mathcal{S}=(G_i)_{i\in \mathbb{N}}$ of pairwise disjoint open sets $G_i \subseteq \mathbb{R}^n$ such that $G_i$ is a scaled copy of $G$ by some $\lambda_i > 0 $; in other words, $G_i$ is congruent (i.e. isometric) to $\lambda_i G$. The sequence $(\lambda_i)_{i\in \mathbb{N}}$ is called the associated scaling sequence of the spray. In applications one has typically $\lambda_i < 1$ and often $\lambda_0 = 1$ so that $G_0$ is equal (or isometric) to $G$. Furthermore, it is meaningful to assume $\displaystyle \sum_{i\in \mathbb{N}}\lambda_i^n< \infty$ to make the volume of $\displaystyle \bigcup_{i\in \mathbb{N}}G_i$ finite, as one deals with the inner tube of $\displaystyle \bigcup_{i\in \mathbb{N}}G_i$ (the inner $\varepsilon$-tube of an open set $A\subseteq \mathbb{R}^n$ is the set of points of $A$ with distance less than $\varepsilon$ to the boundary of $A$).

The most important class of sprays is the self-similar sprays, for which the scaling sequence is of a very special type.

Let $\{r_1, r_2, \ldots , r_J\}$ be a so-called ratio list (i.e. $0 < r_j < 1$ for $j=1,2,\ldots, J$) and consider the formal expression

\begin{equation}\label{moranseries}
\frac{1}{1-(r_1+r_2+\cdots +r_J)}=\sum_{k=0}^\infty \ \sum_{k_1+\cdots +k_J=k}r_1^{k_1}\, r_2^{k_2}\, \ldots  \, r_J^{k_J}.
\end{equation}
   If the scaling sequence $(\lambda_i)_{i\in \mathbb{N}}$  of a spray is given by the terms of the series on the right-hand side of (\ref{moranseries}) for an appropriate ratio list, then the spray is called a self-similar spray.

To give a flavour of tube formulas, we note the following theorem (\cite{LaPeActa}, \cite{LPW}, \cite{DKOUIntelligencer}).

\begin{theorem}
   Let $(G_i)_{i\in \mathbb{N}}$ be a self-similar spray generated by $G\subseteq \mathbb{R}^n$ with a scaling sequence associated with a ratio list $\{r_1, r_2, \ldots , r_J\}$. Assume $G\subseteq \mathbb{R}^n$ to be monophase, i.e. let its inner $\varepsilon$-tube volume function $V_G(\varepsilon)$ be given by
     \begin{eqnarray}
     V_G(\varepsilon)&=& \left\{ \begin{array}{ccc}
     \displaystyle \sum_{i=0}^{n-1} \kappa_i \varepsilon^{n-i} &, & 0\leq \varepsilon \leq g \\
     Vol(G)&,& \varepsilon \geq g \ ,
     \end{array}\right.
   \end{eqnarray}
 where $g$ is the inradius of $G$ (i.e. the supremum of the radii of the balls contained in $G$). Then the volume $V_{\cup G_i}(\varepsilon)$ of the inner $\varepsilon$-tube of $\cup G_i$ is given by the formula
 \[V_{\cup G_i} (\varepsilon)=\sum_{\omega \in \mathfrak{D} \cup \{0,1,2,\dots,n-1\}} {\rm res}(\zeta(s) \, \varepsilon^{n-s};\omega) \quad , \quad \text{ for } \varepsilon < g \, , \]
 where $\zeta(s)=\frac{1}{1-(r_1^s+\cdots + r_J^s)}\displaystyle\sum_{i=0}^n \frac{g^{s-i}}{s-i} \, \kappa_i$ \ (with $\kappa_n=-Vol(G)$) and $\mathfrak{D}$ is the set of complex roots of the equation $r_1^s+\cdots + r_J^s=1$.
\end{theorem}

The motivation behind the notion of the self-similar spray is that they naturally emerge as the ``hollow spaces" in self-similar fractals. For example, if you start with an interval, and construct a Cantor set by deleting successively the open middle thirds of the intervals, the collection of deleted open intervals constitute a $1-$dimensional self-similar spray with a scaling sequence associated with the ratio list $\{\frac{1}{3},\frac{1}{3}\}$, the generator being the first deleted middle third. Likewise, if you start with a triangle and successively delete the open middle fourths of the triangles to obtain in the end a Sierpinski Gasket, then the collection of the deleted open triangles constitute a $2-$dimensional self-similar spray with a scaling sequence associated with the ratio list $\{\frac{1}{2},\frac{1}{2},\frac{1}{2}\}$, the generator being again the first deleted middle fourth.

One can in principle allow any scaling sequence for a spray, but to obtain manageable tube formulas some sort of  restrictions seem (as to yet) to be necessary. To be associated with a ratio list is, for example, such a condition. A weaker condition (called ``subshift of finite-type") was formulated to handle the hollow spaces of graph-directed fractals in \cite{DKOUIntelligencer}, but a more natural approach demands a consideration of sprays with more than one generator, since the hollow spaces of graph-directed fractals are composed of copies of several generators (see Figures~\ref{attractors}-\ref{hollowspaces}).

In Section 2 we propose this more general concept of sprays with several generators. In Section 3 we consider a natural functional equation for inner tube volumes of graph-directed sprays, formulate a multi-dimensional renewal lemma to handle it and establish an inner tube formula for graph-directed sprays as our main result (Theorem~\ref{maintheorem}). In Section 4 we give the proof, taking into account the additional difficulties arising from the presence of the Mauldin-Williams matrix. We give also an explicit sufficient region for $\varepsilon$ on which the main theorem holds.

\section{Graph-Directed Sprays}

Let $\mathcal{G}=(V,E,r)$ be a weighted directed graph with weights $r:E\rightarrow (0,1)$. For an edge $e\in E$, we denote the initial vertex of $e$ by $i(e)$ and the terminal vertex by $t(e)$. For vertices $u,v \in V$, we denote the set of edges from $u$ to $v$ by $E_{uv}$ and the set of edges starting from $u$ by $E_u$. If $E _u \neq \emptyset$ for all $u\in V$, such a graph is called a Mauldin-Williams graph. If any two vertices $u$ and $v$ can be joined by a (directed) path, then the graph is said to be strongly connected. We will generally assume that the Mauldin-Williams graphs be strongly connected.

We define the weight of a path $\alpha=e_1e_2\cdots e_k$ by $r(\alpha)=r(e_1)\cdot r(e_2)\cdot \ldots \cdot r(e_k)$. $\alpha$ is called a path from the vertex $u$ to $v$ if $i(e_1)=u$ and $t(e_k)=v$. We also write $i(\alpha)=u$ and $t(\alpha)=v$. We assign an empty path $\phi_u$ to every vertex $u$ with weight $r(\phi_u)=1$.

Now we define graph-directed sprays.

\begin{definition}\label{DEFgraphdirectedspray}
Let $\mathcal{G}=(V,E,r)$ be a Mauldin-Williams graph and $G_u\, (u \in V)$ be bounded open sets in $\mathbb{R}^n$. A graph-directed spray $\mathcal{S}$ associated with $\mathcal{G}$ and generated by the open sets $G_u\, (u \in V)$ is a collection of pairwise disjoint open sets $G_{\alpha}$ in $\mathbb{R}^n$ (where $\alpha$ is a path in the graph), such that $G_{\alpha}$ is a scaled isometric copy of $G_{t(\alpha)}$ with scaling ratio $r(\alpha)$.
\end{definition}

\begin{remark}
Note that  $G_u$, with $u \in V$, is a generator and $G_\alpha$, where $\alpha$ denotes a path, is a copy of a generator. If $\alpha=\phi_u$ then $G_{\phi_u}$ is a scaled isometric copy of $G_{t(\phi_u)}=G_u$ with scaling ratio $r(\phi_u)=1$, i.e. $G_{\phi_u}$ is an isometric copy of $G_u$.
\end{remark}

\begin{remark}
Note that if $\mathcal{G}$ has only one node this notion reduces to the ordinary notion of a self-similar spray generated by a single open set with a scaling sequence associated with the ratio list consisting of the weights of the loops around the single vertex.
\end{remark}

The spray $\mathcal{S}$ can naturally be decomposed into subcollections \[\mathcal{S}_u=\{G_{\alpha} \, | \, i(\alpha)=u \}.\]
Notice that the subcollection  $\mathcal{S}_u$ is also composed of pairwise disjoint scaled copies of all generators $G_v$ with scaling ratios $r(\alpha)$ for paths $\alpha$ starting at $u$. This decomposition of $\mathcal{S}$ into the subcollections $\mathcal{S}_u$ will prove useful in establishing tube formulas.

The motivation for the definition of graph-directed sprays comes, in analogy to the motivation of self-similar sprays, from hollow spaces of graph-directed fractals. Let us very briefly recall the notion of graph-directed fractals.

Let $\mathcal{G}=(V,E,r)$ be a Mauldin-Williams graph, $(A_u)_{u \in V}$ be a list of complete subsets of $\mathbb{R}^n$ and let $f_e:A_{t(e)}\to A_{i(e)}$ be similarities with similarity ratios $r(e)$. Such an assignment is called a realization of the graph $\mathcal{G}$ in $\mathbb{R}^n$. Given such a realization,  there is a unique list $(K_u)_{u\in V}$ of nonempty compact sets with $K_u \subset A_u \, (u\in V)$ satisfying
\[K_u=\bigcup_{v\in V} \bigcup_{e\in E_{uv}} f_e(K_v)\]
for all $u\in V$ (\cite{EdgarBook}).

In favorable cases the maps $f_e$ can be restricted to the convex hull of the graph-directed ``attractors" $K_{t(e)}$ and these attractors can be imagined to be formed by deleting successively pieces of the convex hull analogous to the construction of the Cantor set or the Sierpinski Gasket by deleting successively pieces of an interval or a triangle. The collection of deleted open pieces will constitute a graph-directed spray in the above defined sense, with so many generators as there are nodes of the graph. Before making this idea precise, it will be best to study an example.

\begin{example}\label{detailedexample}
Consider the Mauldin-Williams graph with $V=\{1,2\}$, with $9$ edges and the corresponding weights as shown in Figure~\ref{graph}.
\begin{figure}[h!]
   {\centering
    \includegraphics[width=0.6\textwidth]{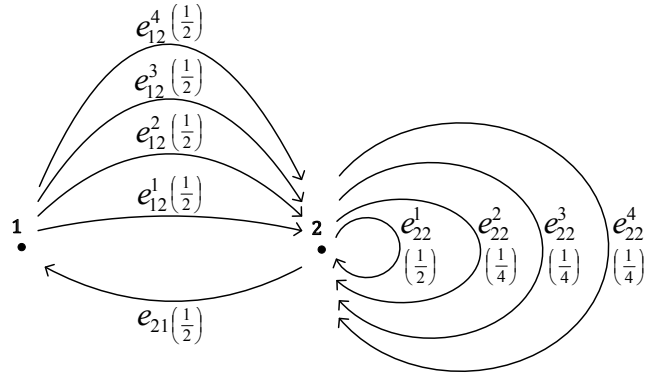}
    \caption{ A Mauldin-Williams graph with 2 nodes and 9 edges (the weights are shown in parenthesis).} \label{graph}
   }
\end{figure}

Let $A_1$ and $A_2$ be the square and the triangle in $\mathbb{R}^2$ as shown in Figure~\ref{similitudes}~(a). The similarities associated with the edges are indicated in Figure~\ref{similitudes}~(b). The graph-directed fractals $K_1$ and $K_2$ of the system are shown in Figure~\ref{attractors}. $(A_1\setminus K_1) \cup (A_2\setminus K_2)$ is a collection of connected open sets which constitute a graph-directed spray with generators $G_1$ and $G_2$ satisfying Definition~\ref{DEFgraphdirectedspray} (see Figures~\ref{generators}-\ref{hollowspaces}). The generators $G_1$ and $G_2$ thereby are defined by $\displaystyle G_u=A_u^{\circ}\setminus \bigcup_{e\in E_u} f_e(A_{t(e)})$ for $u=1,2$.

\begin{figure}[h!]
   {\centering
    \includegraphics[width=0.60\textwidth]{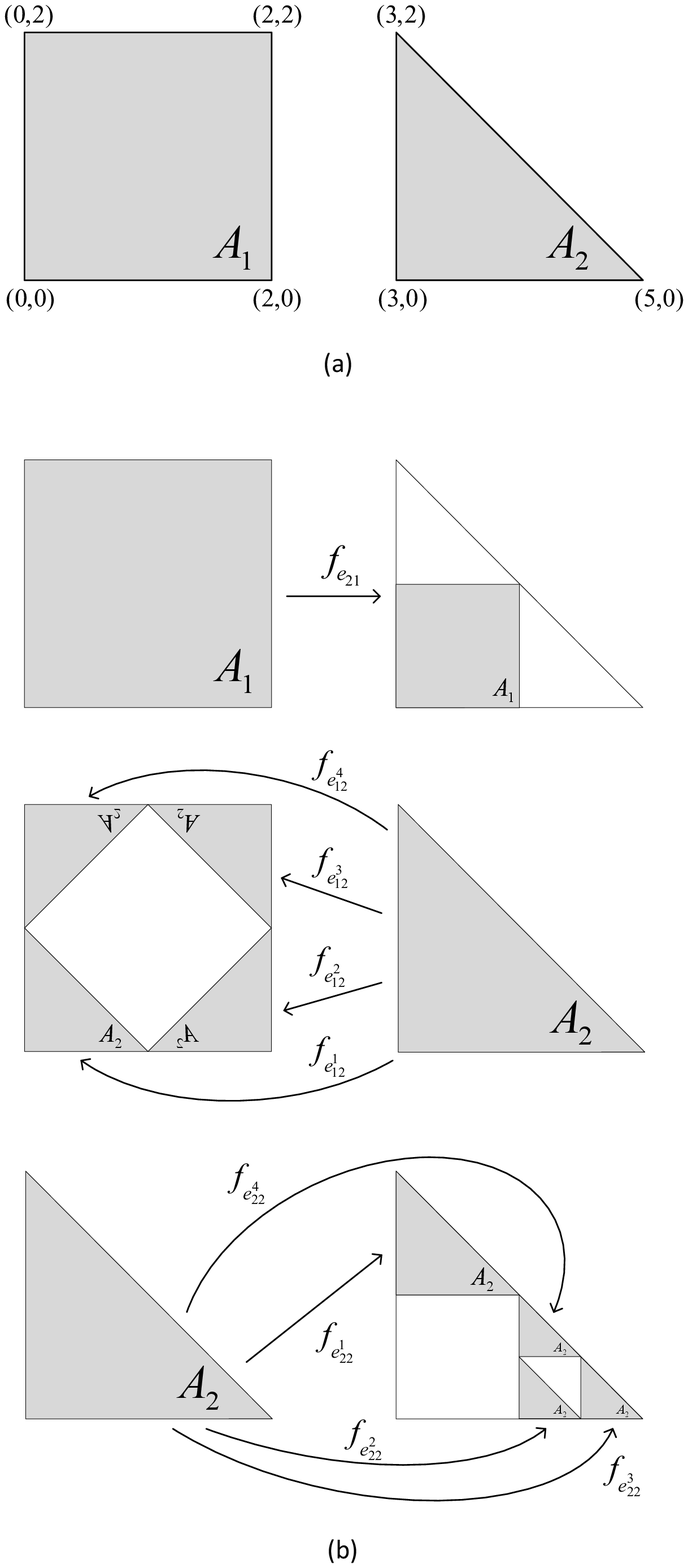}
    \caption{ A realization of the Mauldin-Williams graph of Example~\ref{detailedexample} (shown in \mbox{Figure~\ref{graph}}). (a): The complete spaces associated with the 2 nodes. (b): The similarities associated with the edges.} \label{similitudes}
   }
\end{figure}

\begin{figure}[h!]
   {\centering
    \includegraphics[width=0.65\textwidth]{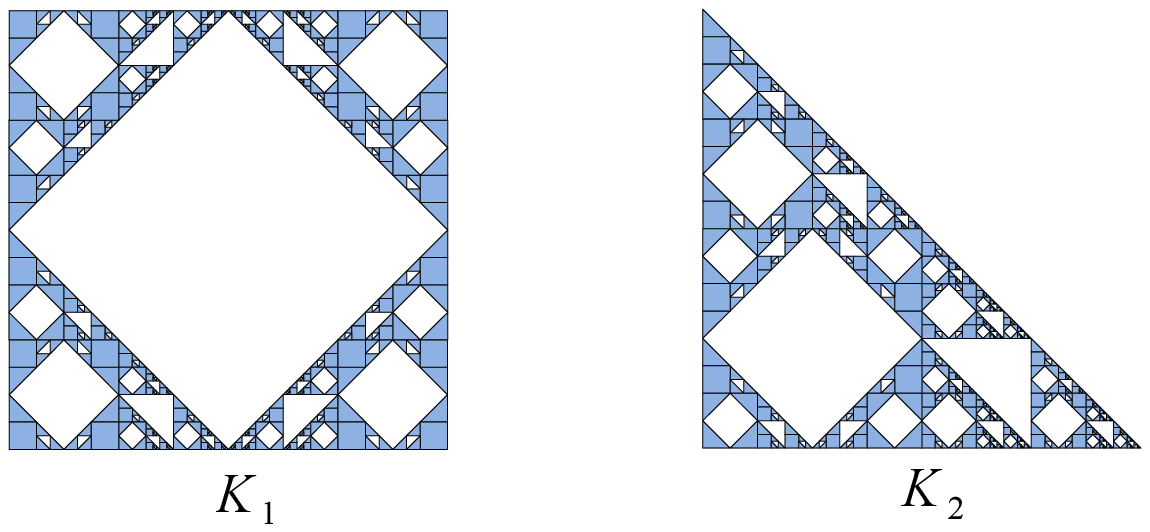}
    \caption{ The attractors of the realization shown in Figure~\ref{similitudes} of the Mauldin-Williams graph of Example~\ref{detailedexample}. } \label{attractors}
   }
\end{figure}

\begin{figure}[h!]
   {\centering
    \includegraphics[width=0.65\textwidth]{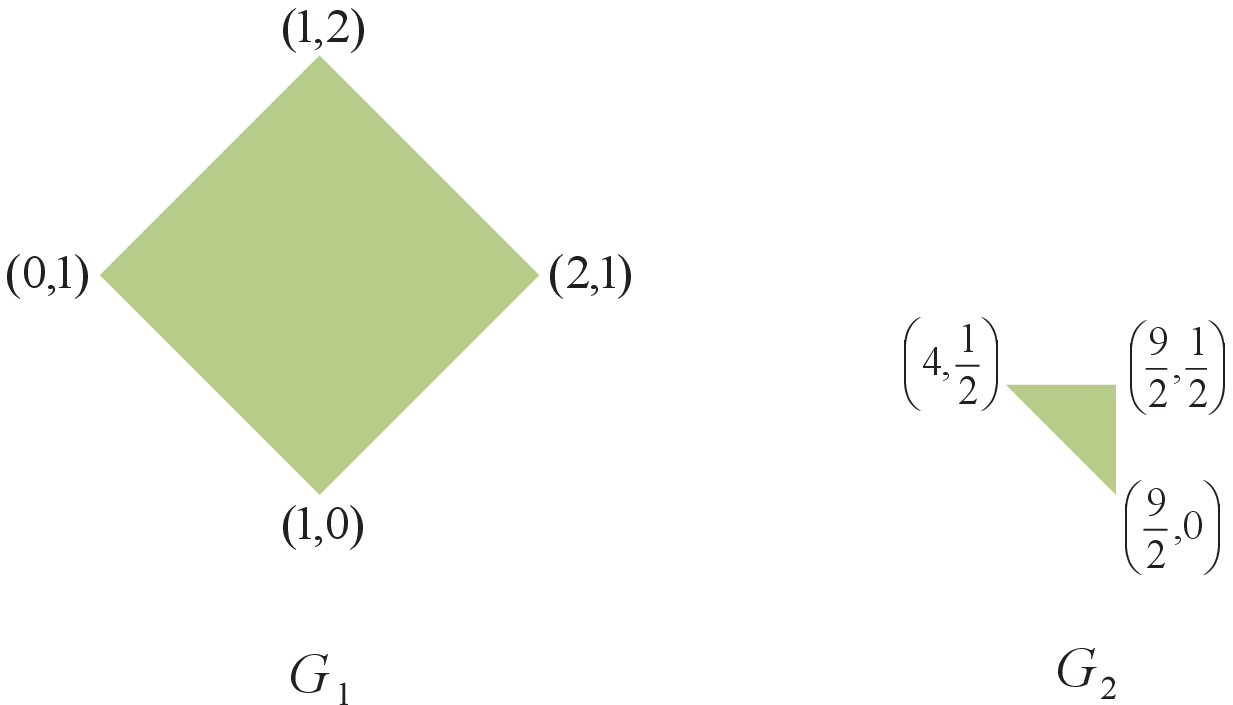}
    \caption{ The generators of the realization shown in Figure~\ref{similitudes} of the Mauldin-Williams graph of Example~\ref{detailedexample}. } \label{generators}
   }
\end{figure}

\begin{figure}[h!]
   {\centering
    \includegraphics[width=0.65\textwidth]{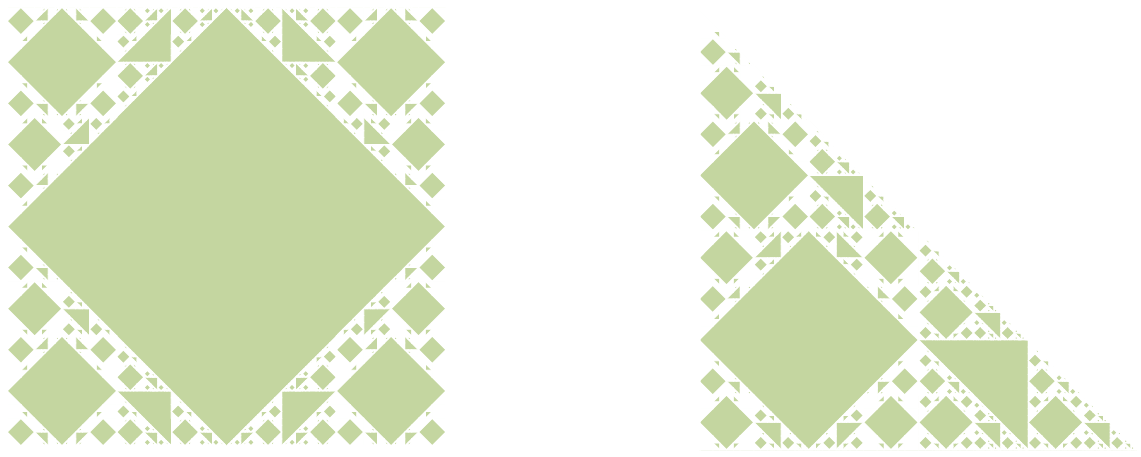}
    \caption{ The hollow spaces of the realization shown in Figure~\ref{similitudes} of the Mauldin-Williams graph of Example~\ref{detailedexample} ($\mathcal{S}_1$ left, $\mathcal{S}_2$ right).} \label{hollowspaces}
   }
\end{figure}

Note that the subcollections $\mathcal{S}_1=\{G_\alpha \, | \, i(\alpha)=1\}$ and $\mathcal{S}_2=\{G_\alpha \, | \, i(\alpha)=2\}$ contain copies of both of $G_1$ and $G_2$ so that we can not view either of them as a spray in the classical sense.

To illustrate the formation of $G_\alpha$ for a path $\alpha$, we give several examples in \mbox{Figures~\ref{exsparys1}}.

\clearpage

\begin{figure}[h!]
   {\centering
    \includegraphics[width=1\textwidth]{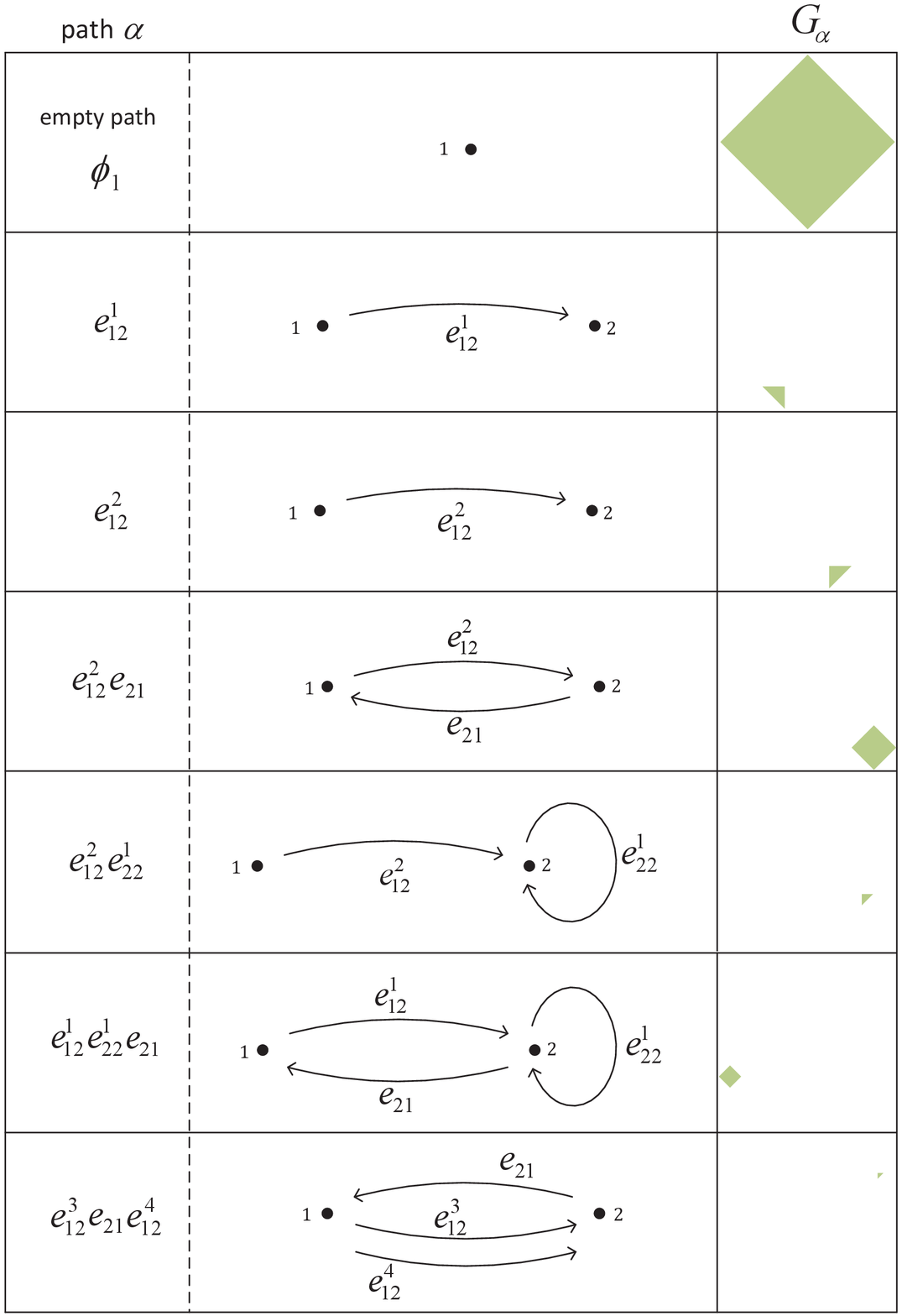}
    \caption{ A few examples of actions of paths on generators for the realization shown in Figure~\ref{similitudes} of the Mauldin-Williams graph of Example~\ref{detailedexample}.} \label{exsparys1}
   }
\end{figure}
\end{example}

\clearpage

From the point of view of our present concern to establish inner tube formulas for graph-directed sprays, the special positions of the scaled copies of the generators are not important as long as they are pairwise disjoint. But if one wishes to compute tube volumes of fractals, one should be careful in relating the tube of the fractal to the inner tube of an associated spray. In the above example the $\varepsilon$-tube volume of $K_1$ can be expressed as the sum of the inner $\varepsilon$-tube volume of the collection $\mathcal{S}_1$ and the outer $\varepsilon$-tube of the square $A_1$ (likewise for $K_2$) (see Figure~\ref{ex1tubes}). But this happy relationship does not hold always as the following example shows.

\begin{figure}[h!]
   {\centering
    \includegraphics[width=0.60\textwidth]{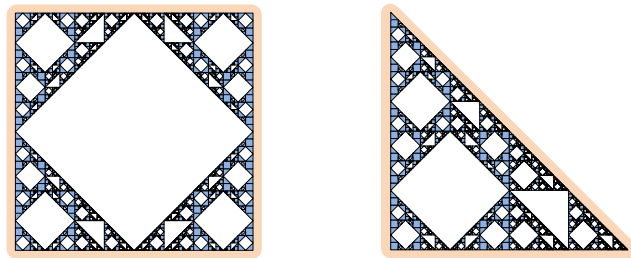}
    \caption{ The $\varepsilon$-tubes of the attractors of Example~\ref{detailedexample} as union of the inner $\varepsilon-$tube of the hollow spaces and the outer $\varepsilon-$tube of the convex hull of the attractor.} \label{ex1tubes}
   }
\end{figure}

\begin{example}\label{exampledifferent}
As the Mauldin-Williams graph we choose the same graph in Figure~\ref{graph} with the only difference that we delete the edge $e_{22}^4$. We choose the same realization of this graph in $\mathbb{R}^2$ discarding the map corresponding to $e_{22}^4$. The emerging graph-directed fractals $L_1$ and $L_2$ and the corresponding graph-directed spray with generators $H_1$ and $H_2$ satisfying Definition~\ref{DEFgraphdirectedspray} are shown in Figures~\ref{attractors2}-\ref{hollowspaces2}.

\begin{figure}[h!]
   {\centering
    \includegraphics[width=0.60\textwidth]{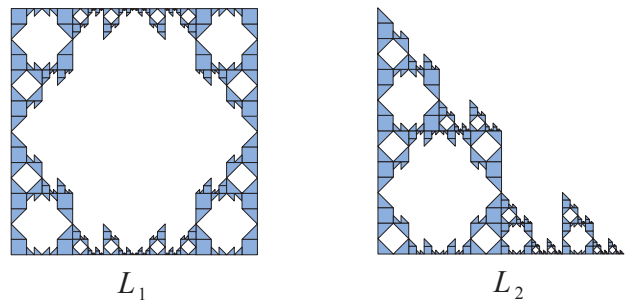}
    \caption{ Attractors of Example~\ref{exampledifferent}} \label{attractors2}
   }
\end{figure}
\begin{figure}[h!]
   {\centering
    \includegraphics[width=0.60\textwidth]{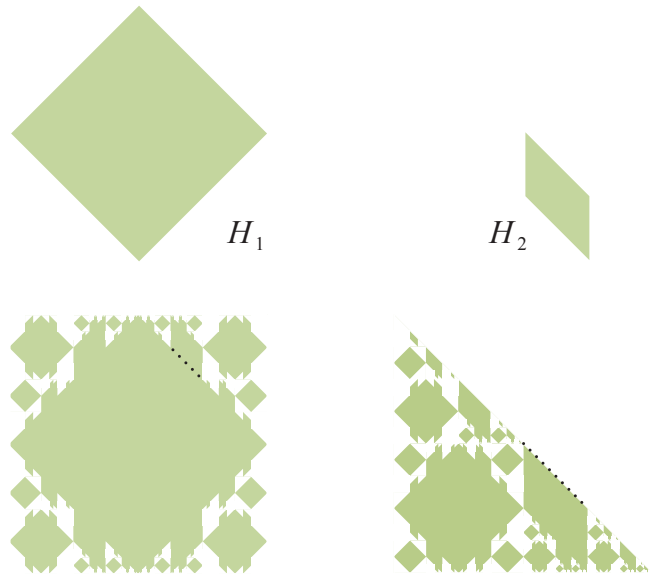}
    \caption{ Generators and the hollow spaces of Example~\ref{exampledifferent} showing that the $\varepsilon-$tube of the attractors needn't be the union of the inner $\varepsilon-$tube of the hollow space and the outer $\varepsilon-$tube of the convex hull of the attractor. } \label{hollowspaces2}
   }
\end{figure}

In this example the $\varepsilon$-tube volume of the graph-directed fractals can not be meaningfully related to the inner $\varepsilon$-tube volume of the graph-directed spray. Two types of problematic boundaries of generator copies are indicated in Figure~\ref{hollowspaces2} by dotted lines.
\end{example}

The simple relationship between the fractal tube volume and the inner spray volume observed in Example~\ref{detailedexample} still remains true for a more general class of graph-directed systems, if the following assumptions hold:

\begin{enumerate}
  \item [i)] ${\rm dim} (C_u)=n$, where $C_u$ is the convex hull $[K_u]$ of $K_u$.

  \item [ii)] (Tileset condition) The open set condition should be satisfied with $O_u=C_u^{\circ}$. We recall that (see \cite{EdgarBook}) the graph-directed system satisfies the open set condition if there exists a list $(O_u)_{u\in V}$ of open sets $O_u \subset \mathbb{R}^n$ such that,
      \begin{enumerate}
      \item for any $e\in E_{uv}$, $f_e(O_v)\subset O_u$,
      \item for any two distinct $e_1,e_2 \in E_u$, $f_{e_1}(O_{t(e_1)})\cap f_{e_2}(O_{t(e_2)})=\emptyset$.
      \end{enumerate}
  \item [iii)] (Nontriviality condition) $\displaystyle C_u^{\circ}\nsubseteq \bigcup_{e\in E_u} f_e(C_{t(e)})$.
  \item [iv)] (Pearse-Winter condition, \cite{PW}) $\partial C_u \subset K_u$.
\end{enumerate}

Now, if we define $\displaystyle G_u=C_u^{\circ} \setminus \bigcup_{e\in E_u} f_e(C_{t(e)})$ then we get a graph-directed spray $\mathcal{S}$ generated by the open sets $(G_u)_{u\in V}$ with $G_\alpha=G_{e_1e_2\cdots e_k}=f_{e_1} f_{e_2}\cdots f_{e_k}(G_{t(e_k)})$ for a path $\alpha$ in the graph $\mathcal{G}$.

Under the above conditions one can compute the $\varepsilon$-tube volume of the graph-directed fractals with the help of the inner tube volume of the graph-directed spray as in Example~\ref{detailedexample}. So we now consider the inner tube volumes for graph-directed sprays in the next section.

\section{Inner Tube Volumes of Graph-Directed Sprays via Functional Equations}

Let $\mathcal{G}=(V,E,r)$ be a Mauldin-Williams graph, $(G_u)_{u\in V}$ bounded open sets in $\mathbb{R}^n$ and $\mathcal{S}$ be a graph-directed spray associated with $\mathcal{G}$ and generated by the open sets $(G_u)_{u\in V}$. Let $\mathcal{S}_u$ be the subcollection of the spray $\mathcal{S}$ corresponding to the paths with initial vertex $u\in V$.

The volume of the inner $\varepsilon$-tube of the collection $\mathcal{S}_u$ satisfies the following functional equation (for all $u\in V$) as one can easily verify:
\begin{equation}\label{functionalequationgenel}
  V_{\mathcal{S}_u}(\varepsilon)=\sum_{v\in V}\sum_{e\in E_{uv}} r_e^n \, V_{\mathcal{S}_v}(\varepsilon/r_e)+ V_{G_u}(\varepsilon).
\end{equation}

Our strategy will be, as in the self-similar case (\cite{DKOUJofG}), to apply the Mellin transform to this functional equation and then try to recover the volume function by applying the inverse Mellin transform. To apply the Mellin transform we need an estimate of  $V_{\mathcal{S}_u}(\varepsilon)$ as $\varepsilon \to 0$. We now formulate a multi-dimensional renewal lemma which will enable us to find such an estimate (For the one-dimensional renewal lemma see \cite{LeVas}).

We recall that for a strongly connected Mauldin-Williams graph the spectral radius of the matrix $A(s)=\left[a_{uv}(s)\right]_{u,v \in V}$ with \mbox{$\displaystyle a_{uv}(s)=\sum_{e\in E_{uv}} r_e^s$} takes the value $1$ for a unique $s_0 \geq 0$, which is called the sim-value of the graph and which we denote by $D$ below (see \cite{EdgarBook}). We will always assume $D<n$.

\begin{remark}
The assumption $D<n$ is in fact equivalent to the condition that the total volume of the graph-directed spray is finite. To see this, one can easily verify that the volumes of the subcollections $\mathcal{S}_u$ can be expressed as follows:
\[[Vol(\mathcal{S}_u)]_{u\in V}=(I+A(n)+A^2(n)+\cdots) \, [Vol(G_u)]_{u \in V},\]
where $[Vol(\mathcal{S}_u)]_{u\in V}$ is a column vector. Note that, the matrix power $A^k(n)$ codes the contribution of paths of length $k$ to the total spray volume.

One can now prove the above claim taking into account that the spectral radius of $A(s)$ is a strictly decreasing function for $s\geq 0$ \cite{EdgarBook}, and for any matrix $A$, each entry of $A^k$ tends to $0$ as $k \to \infty$ if and only if the spectral radius is less than $1$ \cite[Theorem~5.6.12]{HornMatrixBook}.
\end{remark}

\begin{lemma}\label{lemmarenewal}
Let $\mathcal{G}=(V,E,r)$ be a strongly connected Mauldin-Williams graph and
\begin{equation}\label{renewaleq}
h_u(x)= \sum_{v\in V}\sum_{e \in E_{uv}} r_e^D \, h_{v}(x-\log(1/r_e)) \ + \ \psi_u(x) \qquad (u\in V),
\end{equation}
be a system of renewal equations on $\mathbb{R}$, where $D$ is the sim-value of the graph. Assume $\psi_u(x)=O(e^{-\tau |x|})$ for some $\tau>0$. Let $(h_u)_{u\in V}$ be a solution of this system of renewal equations. If $h_u$ tends to $0$ for $x\to -\infty$ for all $u\in V$, then $h_u$ is bounded (for all $u\in V$).
\end{lemma}
\begin{proof}
Let $\displaystyle \gamma=\min_{e\in E}\left\{\log \frac{1}{r_e}\right\}$.

Since the Mauldin-Williams graph $\mathcal{G}$ is strongly-connected, the corresponding Mauldin-Williams matrix $A(s)$ is irreducible and by the Perron-Frobenius theorem, for $s=D$ the spectral radius $1$ is also an eigenvalue with a positive eigenvector $p=(p_u)_{u\in V}$ (with $p_u>0$) so that we have

\begin{equation}\label{morangraphdirected}
p_u=\sum_{v \in V} a_{uv}(D) \, p_v=\sum_{v\in V} \sum_{e\in E_{uv}} r_e^D \, p_v.
\end{equation}

Since the functions \mbox{$h_u \, (u\in V)$} tend to zero for $x\to -\infty$, one can choose $x_0 \in \mathbb{R}$ such that $|h_u(x)| \leq p_u$ for $x\in (-\infty,x_0]$.

Let $x \in [x_0,x_0+\gamma]$. From (\ref{renewaleq}) and (\ref{morangraphdirected})
\begin{eqnarray*}
  |h_u(x)| & \leq & \sum_{v \in V} \sum_{e \in E_{uv}} r_e^D \, p_v \ + \ \sup_{x \in [x_0,x_0+\gamma]}|\psi_u (x)|\\
         & \leq & p_u + \ \sup_{x \in [x_0,x_0+\gamma]}|\psi_u (x)|.
\end{eqnarray*}

By the assumption on $\psi_u$, we can find an $M$ such that $\displaystyle |\psi_u(x)| \leq  p_u \, M \,  e^{-\tau |x|}$ (for all $u \in V$). Hence for $x \in [x_0,x_0+\gamma]$
\begin{equation}\label{renewalrestriction1}
|h_u(x)| \leq p_u \left(1 + M   \sup_{x \in [x_0,x_0+\gamma]}e^{-\tau |x|}\right).
\end{equation}
Since $|h_u(x)| \leq p_u$ for $x\in (-\infty,x_0]$, the inequality (\ref{renewalrestriction1}) holds for all $x\in (-\infty,x_0+\gamma]$.

Now, let $x \in  [x_0+\gamma,x_0+2\gamma]$. As above,
\begin{eqnarray*}
  |h_u(x)|  & \leq & \sum_{v \in V} \sum_{e \in E_{uv}} r_e^D \, p_v \, \left( 1+   M  \sup_{x \in [x_0,x_0+\gamma]}e^{-\tau |x|} \right)  \ + \ \sup_{x \in [x_0+\gamma,x_0+2\gamma]}|\psi_u(x)|\\
  & \leq & p_u \left(1+ M  \sup_{x \in [x_0,x_0+\gamma]}e^{-\tau |x|} \right)  + p_u\, M  \sup_{x \in [x_0+\gamma,x_0+2\gamma]}e^{-\tau |x|}\\
  &    = & p_u \left(1+ M  \sup_{x \in [x_0,x_0+\gamma]}e^{-\tau |x|}  +  M  \sup_{x \in [x_0+\gamma,x_0+2\gamma]}e^{-\tau |x|} \right).
\end{eqnarray*}
The above inequality clearly holds for all $x \in (-\infty,x_0+2\gamma]$.
Repeating the above argument, we see that for  all $x \in \mathbb{R}$,
\[|h_u(x)| \leq  p_u \left(1 +  M   \sum_{k=0}^{\infty} \, \sup_{x \in [x_0+k \gamma,x_0+(k+1)\gamma]}e^{-\tau |x|}\right). \]

This shows that $h_u$ is bounded on $\mathbb{R}$.
\end{proof}

Now we can derive an estimate for $V_{\mathcal{S}_u}(\varepsilon)$ as $\varepsilon \to 0$.

\begin{lemma}\label{lemmaestimate}
 Assume, there exists an $\alpha>0$ such that $V_{G_u}(\varepsilon)=O(\varepsilon^\alpha)$ for $\varepsilon \to 0$. We further assume that $n-\alpha<D<n$, where $D$ is the sim-value of the graph. Then, it holds $V_{\mathcal{S}_u}(\varepsilon)=O(\varepsilon^{n-D})$ as $\varepsilon \to 0$.
\end{lemma}
\begin{remark}
If the generator $G_u$ is monophase or pluriphase (i.e. the volume of the inner $\varepsilon$-tube of $G_u$ is piecewise polynomial), then $\alpha \geq 1$ and the assumption $n-\alpha<D$ reduces to $n-1<D$. In applications to graph-directed fractals, $G_u$ will be mostly monophase or pluriphase, but in graph-directed sprays $G_u$ can be themselves of a fractal nature, so that $\alpha$ could be less than $1$.
\end{remark}
\begin{proof}[Proof of Lemma~\ref{lemmaestimate}]
Let us define $\displaystyle W_{\mathcal{S}_u}(\varepsilon)=\frac{V_{\mathcal{S}_u}(\varepsilon)}{\varepsilon^{n-D}}$ to obtain
\begin{equation}\label{renewaleqnormed}
  W_{\mathcal{S}_u}(\varepsilon)=\sum_{v\in V}\sum_{e\in E_{uv}} r_e^D \, W_{\mathcal{S}_v}(\varepsilon/r_e)+ \frac{V_{G_u}(\varepsilon)}{\varepsilon^{n-D}}
\end{equation}
from the equation (\ref{functionalequationgenel}). Let us now apply the change of variable $\varepsilon=e^{-x}$ in the equation (\ref{renewaleqnormed}). We obtain the following system of renewal equations on $\mathbb{R}$
\begin{equation*}
  h_u(x)=\sum_{v\in V}\sum_{e\in E_{uv}} r_e^D \, h_v(x-\log 1/r_e)+\psi_u(x),
\end{equation*}
where $h_u(x)=W_{\mathcal{S}_u}(e^{-x})$ and $\displaystyle \psi_u(x)= \frac{V_{G_u}(e^{-x})}{e^{-x(n-D)}}$.

Now we have to verify the assumptions of Lemma~\ref{lemmarenewal}.

$\displaystyle h_u(x)=h_u(-\log \varepsilon)=\frac{V_{\mathcal{S}_u}(\varepsilon)}{\varepsilon^{n-D}}$ tends to zero as $x \to -\infty$, i.e. $\varepsilon \to \infty$, since $D<n$ and the volume of the spray is finite.

To check the assumption on $\psi_u$, choose $\tau=\min \{n-D,D-n+\alpha\}$.

By Lemma~\ref{lemmarenewal},  $\displaystyle h_u(x)=\frac{V_{\mathcal{S}_u}(\varepsilon)}{\varepsilon^{n-D}}$ is bounded, so that
$\displaystyle V_{\mathcal{S}_u}(\varepsilon)=O(\varepsilon^{n-D})$.
\end{proof}

We will now apply the Mellin transform to the equation (\ref{functionalequationgenel}) and to this end, it will be convenient to define the auxiliary functions $\displaystyle f_u(\varepsilon)=\frac{V_{\mathcal{S}_u}(\varepsilon)}{\varepsilon^{n}}$, for $u\in V$ (these functions can be viewed as a kind of ``normed'' tube volumes). The system (\ref{functionalequationgenel}) of functional equations transforms into the following system:
\begin{equation}\label{functionalequationnormed}
  f_u(\varepsilon)=\sum_{v\in V}\sum_{e\in E_{uv}}   f_v(\varepsilon/r_e) + \frac{V_{G_u}(\varepsilon)}{\varepsilon^n}.
\end{equation}

Recall that the Mellin transform of a function $f:(0,\infty) \rightarrow \mathbb{R}$ is given by
\[ \mathcal{M}(f)(s) \equiv \widetilde f (s)=\int_0^{\infty} f(x) \, x^{s-1} dx.\]

If this integral exists for some $c\in \mathbb{R}$ and if the function $f$ is continuous at $x \in (0,\infty)$ and of bounded variation in a neighborhood of $x$, then $f(x)$ can be recovered by the inverse Mellin transform (\cite{Titch})
\[\frac{1}{2\pi  \mathbbm{i}}\, \lim_{T \to \infty} \int_{c-\mathbbm{i} T}^{c+\mathbbm{i} T} \widetilde f(s) \, x^{-s} ds.\]

The function $f_u$ is continuous and $f_u(\varepsilon)=O(\varepsilon^{-n})$ as $\varepsilon \to \infty$.
If for some $\alpha>0$, $V_{G_u}(\varepsilon)=O(\varepsilon^{\alpha})$ as $\varepsilon \to 0$ and $n-\alpha<D<n$, then
by Lemma~\ref{lemmaestimate}, $f_u(\varepsilon)=O(\varepsilon^{-D})$ as $\varepsilon \to 0$. So the integral
\[\int_0^\infty f_u(\varepsilon) \, \varepsilon^{s-1} \, d\varepsilon\]
exists for any $s$ with $D<{\rm Re}(s)<n$. Likewise,  the integral
\[\int_0^\infty \frac{V_{G_u}(\varepsilon)}{\varepsilon^n} \, \varepsilon^{s-1} \, d\varepsilon\]
exists for $n-\alpha<{\rm Re}(s)<n$. We can then take the Mellin transform of (\ref{functionalequationnormed}) to obtain
\begin{equation*}
  \widetilde{f_u}(s)=\sum_{v\in V}\left(\sum_{e\in E_{uv}} r_e^s\right)  \widetilde{f_v}(s) + \mathcal{M}\left(\frac{V_{G_u}(\varepsilon)}{\varepsilon^n}\right)(s)
\end{equation*}
for $D<{\rm Re}(s)<n$.
This system can also be written as a matrix equation
\[
F(s)=A(s)F(s)+\Phi(s) \quad (D<{\rm Re}(s)<n)\, ,
\]
where $F(s)$ is the column vector $\left[\widetilde{f_u}(s)\right]_{u\in V}$, $\displaystyle \Phi(s)=\left[  \mathcal{M}\left(\frac{V_{G_u}(\varepsilon)}{\varepsilon^n}\right)(s) \right]_{u\in V}$  and $A(s)$ is the Mauldin-Williams matrix.

\begin{lemma}\label{matrixinvertion}
For ${\rm Re}(s)>D$, the matrix $I-A(s)$ is invertible, so that it holds
\begin{equation}\label{matrixinversitioneqnum}
F(s)=\left[I-A(s)\right]^{-1} \Phi(s)
\end{equation}
for $D<{\rm Re}(s)<n$.
\end{lemma}
\begin{proof}
This is a consequence of some well-known results from matrix algebra. For $s\in \mathbb{R}$, $s>D$, the spectral radius $\rho(A(s))$ is less than $1$ (\cite{EdgarBook}). Then by \mbox{\cite[Theorem~5.6.12]{HornMatrixBook}} $\displaystyle \lim_{k\to \infty} A^k(s)=0$ entry-wise. For arbitrary $s\in \mathbb{C}$ with ${\rm Re}(s)>D$, we have
\[
|a_{uv}(s)|=\left| \sum_{e\in E_{uv}} r_e^s \right|\leq \sum_{e\in E_{uv}} r_e^{{\rm Re}(s)}=a_{uv}({\rm Re}(s)).
\]
This holds for the entries of $A^k(s)$ and $A^k({\rm Re}(s))$ also, giving $\displaystyle \lim_{k\to \infty} A^k(s)=0$ entry-wise. We then have by \mbox{\cite[Theorem~5.6.12]{HornMatrixBook}}, $\rho(A(s))<1$, and thus $I-A(s)$ is invertible.
\end{proof}

We can write the matrix equation (\ref{matrixinversitioneqnum}) also as follows:
\begin{equation*}
\left[\widetilde{f_u}(s)\right]_{u\in V}=\dfrac{1}{\det(I-A(s))} \, \left[{\rm adj}(I-A(s))\right]_{uv} \, \left[\mathcal{M}\left(\dfrac{V_{G_u}(\varepsilon)}{\varepsilon^n}\right) (s)\right]_{u\in V},
\end{equation*}
or
\begin{equation}\label{matrixequationmellin}
\widetilde{f_u}(s)=\dfrac{1}{\det(I-A(s))} \, \sum_{v\in V} {\rm adj}(I-A(s))_{uv} \, \mathcal{M}\left(\dfrac{V_{G_v}(\varepsilon)}{\varepsilon^n}\right) (s) \quad \mbox{ for all $u \in V$}.
\end{equation}

We now apply the inverse Mellin transform to the equation (\ref{matrixequationmellin}): For $D<c<n$,
\begin{eqnarray*}
f_u(\varepsilon)&=&\frac{1}{2\pi\mathbbm{i}}\,  \int_{c-\mathbbm{i} \infty}^{c+\mathbbm{i} \infty} \widetilde{f_u}(s) \, \varepsilon^{-s} ds\\
&=& \frac{1}{2\pi\mathbbm{i}}\, \sum_{v\in V} \left(
\int_{c-\mathbbm{i} \infty}^{c+\mathbbm{i} \infty} \dfrac{{\rm adj}(I-A(s))_{uv}}{\det(I-A(s))} \, \mathcal{M}\left(\dfrac{V_{G_v}(\varepsilon)}{\varepsilon^n}\right) (s) \, \varepsilon^{-s} \, ds
\right).
\end{eqnarray*}

\begin{definition}
Let $\mathcal{G}=(V,E,r)$ be a Mauldin-Williams graph, $G_u \, (u\in V)$ bounded open sets in $\mathbb{R}^n$ and  $\mathcal{S}$ a graph-directed spray associated with $\mathcal{G}$ and generated by the open sets $G_u$.
Let $A(s)=\left[a_{uv}(s)\right]_{u,v \in V}$ with \mbox{$\displaystyle a_{uv}(s)=\sum_{e\in E_{uv}} r_e^s$} be the Mauldin-Williams matrix of the graph. We define the geometric zeta function of the graph-directed spray with respect to the node $u\in V$ as follows:
\[
\zeta_u(s)=\sum_{v\in V}  \dfrac{{\rm adj}(I-A(s))_{uv}}{\det(I-A(s))} \,\mathcal{M}\left(\dfrac{V_{G_v}(\varepsilon)}{\varepsilon^n}\right) (s),
\]
for $D<{\rm Re}(s)<n$ where $D$ is the sim-value of the Mauldin-Williams graph.
\end{definition}

$f_u(\varepsilon)$ can now be expressed as
\[
f_u(\varepsilon)=\frac{1}{2\pi\mathbbm{i}}\,  \int_{c-\mathbbm{i} \infty}^{c+\mathbbm{i} \infty} \zeta_u(s) \, \varepsilon^{-s} ds.
\]

At this point, we need some assumptions about the inner tube volumes of the generators to manipulate this expression further. We assume that the generators are monophase or pluriphase. Note that in this case the geometric zeta function which is analytic for $D<{\rm Re}(s)<n$ can be extended meromorphically to the whole plane $\mathbb{C}$, as can be seen from the explicit expressions for the Mellin transforms given in the Remarks~\ref{remark1maintheorem}-\ref{remark2maintheorem} below.

We can now express our main result as follows:

\begin{theorem}\label{maintheorem}
Let $\mathcal{G}=(V,E,r)$ be a Mauldin-Williams graph, $G_u \, (u\in V)$ bounded open sets in $\mathbb{R}^n$ and  $\mathcal{S}$ a graph-directed spray associated with $\mathcal{G}$ and generated by the open sets $G_u$.
Let $A(s)$ be the matrix of the Mauldin-Williams graph.

We assume the generators $G_u$ to be monophase or pluriphase. We furthermore assume that the sim-value $D$ of the Mauldin-Williams graph satisfies $n-1<D<n$.

Then for small $\varepsilon$, the volume of the inner $\varepsilon-$tube of the graph-directed spray  $\mathcal{S}$ can be expressed pointwise as the following residue formula:
\[V_\mathcal{S} (\varepsilon)=\sum_{u\in V} \sum_{\omega \in \mathfrak{D} \cup \{0,1,2,\dots,n-1\}} {\rm res}(\zeta_u(s)\, \varepsilon^{n-s};\omega),
\]
where $\mathfrak{D}$ is the set of zeros of $\det(I-A(s))$, which we call the complex dimensions of the graph-directed spray (For an exact bound for $\varepsilon$, see the last paragraph of the proof).
\end{theorem}

\begin{remark}\label{remark1maintheorem}
If the generator $G_u$ is monophase with tube formula
 \begin{eqnarray*}
     V_{G_u}(\varepsilon)&=& \left\{ \begin{array}{ccc}
     \displaystyle \sum_{i=0}^{n-1} \kappa_i^u \, \varepsilon^{n-i} &, & 0\leq \varepsilon \leq g_u \\
     Vol(G_u)&,& \varepsilon \geq g_u \ ,
     \end{array}\right.
   \end{eqnarray*}
then
\[
\mathcal{M}\left(\dfrac{V_{G_u}(\varepsilon)}{\varepsilon^n}\right) (s)=\sum_{i=0}^{n} \kappa_i^u \, \dfrac{g_u^{s-i}}{s-i}\, ,
\]
where $\kappa_n^u=-Vol(G_u)$.
\end{remark}

\begin{remark}\label{remark2maintheorem}
If the generator $G_u$ is pluriphase, let us assume that it has the tube formula
 \begin{eqnarray*}
     V_{G_u}(\varepsilon)&=& \left\{ \begin{array}{ccc}
     \displaystyle \sum_{i=0}^{n} \kappa_i^{m,u} \, \varepsilon^{n-i} &, & g_{m-1,u}\leq \varepsilon \leq g_{m,u}\, , \ m=1,2,\ldots,M_u \\
     Vol(G_u)&,& \varepsilon \geq g_u \ ,
     \end{array}\right.
   \end{eqnarray*}
where $g_{0,u}=0$, $g_{M_u,u}=g_u$ ($g_u$ the inradius of $G_u$) and $\kappa_n^{1,u}=0$. It will be more convenient to write above formula as
\[
 V_{G_u}(\varepsilon)=\sum_{i=0}^{n} \kappa_i^{m,u} \, \varepsilon^{n-i} \ \mbox{ for } g_{m-1,u}\leq \varepsilon \leq g_{m,u}\, , \ m=1,2,\ldots,M_u+1\, ,
\]
where we set $\kappa_i^{M_u+1,u}=0$ for $i=0,1,\ldots,n-1$, $\kappa_n^{M_u+1,u}=Vol(G_u)$ and $g_{M_u+1,u}=\infty$. Then
\[
\mathcal{M}\left(\dfrac{V_{G_u}(\varepsilon)}{\varepsilon^n}\right) (s)=\sum_{m=1}^{M_u} \sum_{i=0}^{n} \left(\kappa_i^{m,u}- \kappa_i^{m+1,u}\right) \dfrac{g_{m,u}^{s-i}}{s-i}\,.\]
\end{remark}

\begin{example}[Example~\ref{detailedexample} continued]
The volumes of the inner $\varepsilon-$neighborhood of $G_1$ and $G_2$ are given by the following functions:
\[
V_{G_1}(\varepsilon)=\left\{\begin{array}{cll}
4\sqrt{2}\varepsilon-4\varepsilon^2&,&\varepsilon\leq \frac{\sqrt 2}{2}\\
2&,&\varepsilon\geq \frac{\sqrt 2}{2}
\end{array}\right.
\]
\[
V_{G_2}(\varepsilon)=\left\{\begin{array}{cll}
\left(\frac{2+\sqrt{2}}{2}\right)\varepsilon-(3+2\sqrt2)\varepsilon^2&,&\varepsilon \leq\frac{2-\sqrt2}{4}\\
\frac{1}{8}&,&\varepsilon\geq \frac{2-\sqrt 2}{4}
\end{array}\right.
\]
The corresponding Mauldin-Williams matrix of the graph is
\[
A(s)=\left(
\begin{array}{ccc}0&4\,\dfrac{1}{2^s}\\
\dfrac{1}{2^s}&\dfrac{1}{2^s}+3 \, \dfrac{1}{4^s}\end{array}
\right)
\]
and the sim-value of the graph is $D=\log_2\left(\frac{\sqrt{29}+1}{2}\right)$. The complex dimensions of the graph-directed spray are given by
\[
\left\{\log_2\left(\frac{\sqrt{29}+1}{2}\right)+\mathbbm{i}kp \, | \, k\in \mathbb{Z} \right\} \cup
\left\{\log_2\left(\frac{\sqrt{29}-1}{2}\right)+\mathbbm{i}\left(k+\frac{1}{2}\right)p \, | \, k\in \mathbb{Z}\right\},
\]
where $p=\dfrac{2\pi}{\ln2}$.
Using Theorem~\ref{maintheorem}, we obtain the volume of the $\varepsilon-$neighborhood of $\mathcal{S}$ as
\begin{eqnarray*}
V_\mathcal{S} (\varepsilon)&=& \frac{4}{7}\varepsilon^2-\frac{28}{5}\sqrt 2\varepsilon+\Sigma_1+\Sigma_2 + \frac{2}{7}(3+2\sqrt2)\varepsilon^2-\frac{3}{5}(2+\sqrt2)\varepsilon+\Sigma_3+\Sigma_4\, ,
\end{eqnarray*}
where
\begin{eqnarray*}
\Sigma_1&=&\frac{\varepsilon^{2-(D+\mathbbm{i}kp)}}{\sqrt{29} \ln 2} \left[\frac{2\sqrt{29}-2}{7}\left(-4\frac{(\frac{\sqrt 2}{2})^{D+\mathbbm{i}kp}}{D+\mathbbm{i}kp}+4\sqrt 2 \frac{(\frac{\sqrt 2}{2})^{D-1+\mathbbm{i}kp}}{D-1+\mathbbm{i}kp}-2\frac{(\frac{\sqrt 2}{2})^{D-2+\mathbbm{i}kp}}{D-2+\mathbbm{i}kp}\right)\right.\\
&& \left.+4\left( -(3+2\sqrt 2)\frac{(\frac{2-\sqrt 2}{4})^{D+\mathbbm{i}kp}}{D+\mathbbm{i}kp}+\frac{2+\sqrt 2}{2} \frac{(\frac{2-\sqrt 2}{4})^{D-1+\mathbbm{i}kp}}{D-1+\mathbbm{i}kp}-\frac{1}{8}\frac{(\frac{2-\sqrt2}{4})^{D-2+\mathbbm{i}kp}}{D-2+\mathbbm{i}kp}\right)\right]
\end{eqnarray*}
\begin{eqnarray*}
\Sigma_2&=&\frac{\varepsilon^{2-(D'+\mathbbm{i}(k+\frac{1}{2})p)}}{\sqrt{29} \ln 2} \times \\
&& \left[\frac{2\sqrt{29}+2}{7}\left(-4\frac{(\frac{\sqrt 2}{2})^{D'+\mathbbm{i}(k+\frac{1}{2})p}}{D'+\mathbbm{i}(k+\frac{1}{2})p}+4\sqrt 2 \frac{(\frac{\sqrt 2}{2})^{D'-1+\mathbbm{i}(k+\frac{1}{2})p}}{D'-1+\mathbbm{i}(k+\frac{1}{2})p}-2\frac{(\frac{\sqrt 2}{2})^{D'-2+\mathbbm{i}(k+\frac{1}{2})p}}{D'-2+\mathbbm{i}(k+\frac{1}{2})p}\right)\right.\\
&& \left.-4\left(-(3+2\sqrt 2)\frac{(\frac{2-\sqrt 2}{4})^{D'+\mathbbm{i}(k+\frac{1}{2})p}}{D'+\mathbbm{i}(k+\frac{1}{2})p}+\frac{2+\sqrt 2}{2} \frac{(\frac{2-\sqrt 2}{4})^{D'-1+\mathbbm{i}(k+\frac{1}{2})p}}{D'-1+\mathbbm{i}(k+\frac{1}{2})p}-\frac{1}{8}\frac{(\frac{2-\sqrt 2}{4})^{D'-2+\mathbbm{i}(k+\frac{1}{2})p}}{D'-2+\mathbbm{i}(k+\frac{1}{2})p}\right)\right]
\end{eqnarray*}

\begin{eqnarray*}
\Sigma_3&=&\frac{\varepsilon^{2-(D+\mathbbm{i}kp)}}{\sqrt{29}\ln 2} \times \\
&& \left[ \left(   -4\frac{(\frac{\sqrt 2}{2})^{D+\mathbbm{i}kp}}{D+\mathbbm{i}kp}+4\sqrt 2 \frac{(\frac{\sqrt 2}{2})^{D-1+\mathbbm{i}kp}}{D-1+\mathbbm{i}kp}-2\frac{(\frac{\sqrt 2}{2})^{D-2+\mathbbm{i}kp}}{D-2+\mathbbm{i}kp}
\right) \right.\\
&& \left.+\frac{\sqrt{29}+1}{2}\left(-(3+2\sqrt 2)\frac{(\frac{2-\sqrt 2}{4})^{D+\mathbbm{i}kp}}{D+\mathbbm{i}kp}+\frac{2+\sqrt 2}{2} \frac{(\frac{2-\sqrt2}{4})^{D-1+\mathbbm{i}kp}}{D-1+\mathbbm{i}kp}-\frac{1}{8}\frac{(\frac{2-\sqrt2}{4})^{D-2+\mathbbm{i}kp})}{D-2+\mathbbm{i}kp}\right)\right]
\end{eqnarray*}
\begin{eqnarray*}
\Sigma_4&=&-\frac{\varepsilon^{2-(D'+\mathbbm{i}(k+\frac{1}{2})p)}}{\sqrt{29}\ln 2} \times\\
&&\left[\left(-4\frac{(\frac{\sqrt 2}{2})^{D'+\mathbbm{i}(k+\frac{1}{2})p}}{D'+\mathbbm{i}(k+\frac{1}{2})p}+4\sqrt 2 \frac{(\frac{\sqrt 2}{2})^{D'-1+\mathbbm{i}(k+\frac{1}{2})p}}{D'-1+\mathbbm{i}(k+\frac{1}{2})p}-2\frac{(\frac{\sqrt 2}{2})^{D'-2+\mathbbm{i}(k+\frac{1}{2})p}}{D'-2+\mathbbm{i}(k+\frac{1}{2})p}\right)\right.\\
&
-&\left.\frac{\sqrt{29}+1}{2}\left(-(3+2\sqrt 2)\frac{(\frac{2-\sqrt 2}{4})^{D'+\mathbbm{i}(k+\frac{1}{2})p}}{D'+\mathbbm{i}(k+\frac{1}{2})p}+\frac{2+\sqrt 2}{2} \frac{(\frac{2-\sqrt 2}{4})^{D'-1+\mathbbm{i}(k+\frac{1}{2})p}}{D'-1+\mathbbm{i}(k+\frac{1}{2})p}-\frac{1}{8}\frac{(\frac{2-\sqrt 2}{4})^{D'-2+\mathbbm{i}(k+\frac{1}{2})p}}{D'-2+\mathbbm{i}(k+\frac{1}{2})p}\right)\right].
\end{eqnarray*}

\end{example}

\bigskip

\section{Proof of the Main Theorem (Theorem~\ref{maintheorem})}

\noindent Since $\displaystyle V_\mathcal{S}(\varepsilon)=\varepsilon^n \sum_{u\in V} f_u(\varepsilon)$ and
\begin{eqnarray*}
f_u(\varepsilon)&=&\frac{1}{2\pi\mathbbm{i}}\, \sum_{v\in V} \left(
\int_{c-\mathbbm{i} \infty}^{c+\mathbbm{i} \infty} \dfrac{{\rm adj}(I-A(s))_{uv}}{\det(I-A(s))} \, \mathcal{M}\left(\dfrac{V_{G_v}(\varepsilon)}{\varepsilon^n}\right) (s) \, \, \varepsilon^{-s} \, ds
\right)\\
&=&\frac{1}{2\pi\mathbbm{i}}\,  \int_{c-\mathbbm{i} \infty}^{c+\mathbbm{i} \infty} \zeta_u(s) \, \varepsilon^{-s} ds\, ,
\end{eqnarray*}
we have to evaluate the integral on the right hand side. As is well-known there is a general procedure to evaluate this integral by applying the residue theorem. In the present case of graph-directed sprays however, the $\det(I-A(s))$ in the denominator calls for a more cautious treatment in order to be able to give an explicit validity range of the tube formula for small $\varepsilon$.

The above sum consists of integrals of the type
\[
\frac{1}{2\pi\mathbbm{i}}\,
\int_{c-\mathbbm{i} \infty}^{c+\mathbbm{i} \infty} \dfrac{{\rm adj}(I-A(s))_{uv}}{\det(I-A(s))} \, \dfrac{g^{s-i}}{s-i}\, \varepsilon^{-s} \, ds \quad (i=0,1,\ldots,n),
\]
for $n-1<D<c<n$. Recall that $\det(I-A(s))$ is non-zero for ${\rm Re}(s)>D$.

We first note that there exists a vertical strip containing all the zeros of $\det(I-A(s))$. To see this, notice that $\det(I-A(s))$ can be expressed as a sum
\[
1+\sum_\alpha p_\alpha^s - \sum_\beta q_\beta^s
\]
with $0<p_\alpha, q_\beta<1$ since the entries of the matrix $A(s)$ are of the form $\sum r_e^s$ for \mbox{$0<r_e<1$}. As ${\rm Re}(s) \to -\infty$, the smallest of $p_\alpha, q_\beta$ will dominate and avoid \mbox{$\det(I-A(s))$} to vanish. We choose a $c_l<0$ such that $|\det(I-A(s))|>\delta$ for some $\delta>0$ and for all ${\rm Re}(s)\leq c_l$.

We will choose a sequence $(\tau_j)_{j\in \mathbb{N}}\to\infty$ such that $|\det(I-A(s))|$ will be uniformly away from zero on the line segments $c_l \leq {\rm Re}(s) \leq c$,  ${\rm Im}(s)=\pm \tau_j$.

\begin{lemma}
There exists an increasing sequence $(\tau_j)_{j\in \mathbb{N}}$, $\tau_j\to\infty$ and a $K>0$ such that $|\det(I-A(s))|>K$ for $c_l \leq {\rm Re}(s) \leq c$  and ${\rm Im}(s)=\pm \tau_j$.
\end{lemma}
\begin{proof}
Being an entire function, $\det(I-A(s))$ has isolated zeros and we can choose $\tau_1>0$ such that there are no zeros on the segment $[c_l,c]\times \{\tau_1\}$. Let $2K$ be the minimum of $|\det(I-A(s))|$ on this segment.

To construct $\tau_2$, we need the following lemma.

\begin{lemma}[Dirichlet Lemma, \cite{Bateman}]\label{dirichletlemma}
Let $M,N \in \mathbb{N}$ and $T>0$. Let $a_1,a_2,\ldots,a_N$ be real numbers. There exists a real number $h\in [T,T\,M^N]$ such that
\[
\|  a_i \, h \|  \leq \dfrac{1}{M} \quad (1\leq i \leq N).
\]
Here $\|  \cdot \| $ denotes ``the distance to the nearest integer'' function on $\mathbb{R}$.
\end{lemma}

The idea for choosing $\tau_2=\tau_1 + h$ will be the following. We want to arrange $h$ such that
\[
| \det(I-A(s+ih))-\det(I-A(s))|<K
\]
for $c_l \leq {\rm Re}(s) \leq c$. Then the minimum of
$| \det(I-A(s+ih))|$ will be greater than $K$. Since
$\displaystyle \det(I-A(s))=1+\sum_\alpha p_\alpha^s - \sum_\beta q_\beta^s$, we have
\begin{eqnarray*}
|\det(I-A(s+ih))-\det(I-A(s))|&=&
\left| \sum_\alpha \left( p_\alpha^{s+ih} - p_\alpha^s \right) -  \sum_\beta \left(q_\beta^{s+ih}- q_\beta^s \right) \right|\\
&\leq& \sum_\alpha \left| p_\alpha^{s+ih} - p_\alpha^s \right|+  \sum_\beta \left| q_\beta^{s+ih}- q_\beta^s \right|\\
&=& \sum_\alpha p_\alpha^{{\rm Re}(s)} \left| p_\alpha^{ih} - 1 \right|+  \sum_\beta q_\beta^{{\rm Re}(s)} \left| q_\beta^{ih}- 1 \right|\, .
\end{eqnarray*}
Since the number of terms and $p_\alpha^{{\rm Re}(s)}, q_\beta^{{\rm Re}(s)}$ are bounded, it will be enough to make the factors $\left| p_\alpha^{ih} - 1 \right|, \left| q_\beta^{ih}- 1 \right|$  small enough. To realize this, we can apply the Dirichlet Lemma (Lemma~\ref{dirichletlemma}) to make $ p_\alpha^{ih}=e^{ih \ln p_\alpha}$ and $ q_\beta^{ih}=e^{ih \ln q_\beta}$ close enough to $1$, by choosing $\displaystyle \left\| \dfrac{h \ln p_\alpha}{2\pi} \right\|$ and $\left\|  \dfrac{h \ln q_\beta}{2\pi} \right\|$ small enough.

As we have control on choosing $h$ on any range, we can repeat this procedure to get a sequence of segments $[c_l,c]\times\{\tau_j\}$ on all of which $|\det(I-A(s))|$ is bounded below by $K$.
\end{proof}

Let us now consider the rectangles $R_j=[c_l,c]\times[-\tau_j,\tau_j]$ and denote its oriented edges by $L_{1,j}, L_{2,j}, L_{3,j}, L_{4,j}$ as shown in Figure~\ref{orientedlines}. We will show that for small enough $\varepsilon$ the integrals on $L_{2,j}, L_{3,j}$ and $L_{4,j}$ will tend to zero as $j\to\infty$ so that by residue theorem we will get the integral on the vertical line at $c$ as a series of residues on the strip $c_l<{\rm Re}(s) < c$:
\[
\frac{1}{2\pi\mathbbm{i}}\,
\int_{c-\mathbbm{i} \infty}^{c+\mathbbm{i} \infty} \dfrac{{\rm adj}(I-A(s))_{uv}}{\det(I-A(s))} \, \dfrac{g^{s-i}}{s-i}\, \varepsilon^{-s} \, ds
=
\sum_{\omega \in \mathfrak{D} \cup \{i\}} {\rm res}\left( \dfrac{{\rm adj}(I-A(s))_{uv}}{\det(I-A(s))} \, \dfrac{g^{s-i}}{s-i}\,   ;\omega\right)
\]
for $i=0,1,\ldots, n-1$. For $i=n$ the same formula holds with the only difference that the residues are taken on $\mathfrak{D}$.

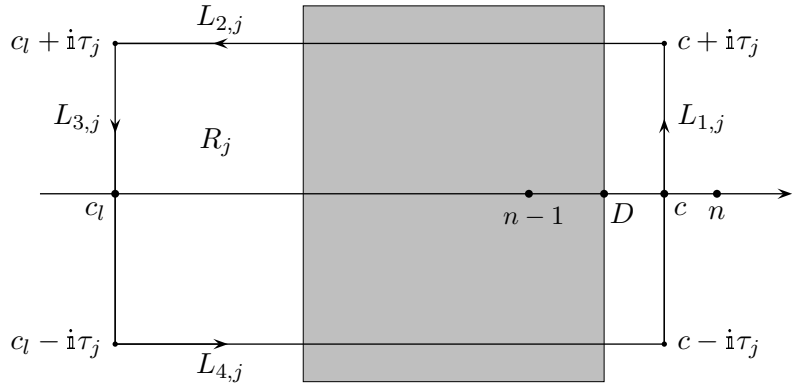
\begin{figure}[h]
\begin{center}
\begin{pspicture}(-10.25,-2.7)(0.25,2.7)
\pspolygon[linewidth=0.01pt,fillstyle=solid,fillcolor=lightgray](-2.5,-2.5)(-2.5,2.5)(-6.5,2.5)(-6.5,-2.5)
\psline[linewidth=0.5pt,arrowsize=4pt]{->}(-10,0)(0,0)
\psline[linewidth=0.5pt](-1.7,-2)(-1.7,2)
\psline[linewidth=0.5pt,arrowsize=4pt]{->}(-1.7,-2)(-1.7,1)
\psline[linewidth=0.5pt](-9,-2)(-9,2)
\psline[linewidth=0.5pt,arrowsize=4pt]{-<}(-9,-2)(-9,1)
\psline[linewidth=0.5pt,arrowsize=4pt](-9,2)(-1.7,2)
\psline[linewidth=0.5pt,arrowsize=4pt]{-<}(-9,2)(-7.5,2)
\psline[linewidth=0.5pt,arrowsize=4pt](-9,-2)(-1.7,-2)
\psline[linewidth=0.5pt,arrowsize=4pt]{->}(-9,-2)(-7.5,-2)
\uput[d](-1,0){$n$}
\uput[d](-3.5,0){ \small$ n-1$}
\uput[dr](-2.55,0){$D$}
\uput[dl](-9,0){$c_{l}$}
\uput[dr](-1.7,0){$c$}
\uput[dr](-8,1){$R_j$}
\uput[r](-1.7,-2){$c-\mathbbm{i}\tau_j$}\qdisk(-1.7,-2){1pt}
\uput[r](-1.7,2){$c+\mathbbm{i}\tau_j$}\qdisk(-1.7,2){1pt}
\uput[l](-9,-2){$c_{l}-\mathbbm{i}\tau_j$}\qdisk(-9,-2){1pt}
\uput[l](-9,2){$c_{l}+\mathbbm{i} \tau_j$}\qdisk(-9,2){1pt}
\uput[u](-7.6,1.95){$L_{2,j}$}
\uput[d](-7.6,-1.95){$L_{4,j}$}
\uput[l](-9,1){$L_{3,j}$}
\uput[r](-1.7,1){$L_{1,j}$}
\qdisk(-9,0){1.5pt}
\qdisk(-3.5,0){1.5pt}
\qdisk(-2.5,0){1.5pt}
\qdisk(-1.7,0){1.5pt}
\qdisk(-1,0){1.5pt}
\end{pspicture}
\end{center}
\caption{The  strip containing the poles of the geometric zeta function $\zeta_u(s)$ and the rectangle $R_j$ with the oriented boundary segments $L_{1,j},L_{2,j},L_{3,j},L_{4,j}$ used in the proof.} \label{orientedlines}
\end{figure}

First consider the integral on $L_{2,j}=t+\mathbbm{i}\tau_j, \, c_l\leq t \leq c$.
\begin{eqnarray*}
\left | \int_{L_{2,j}} \dfrac{{\rm adj}(I-A(s))_{uv}}{\det(I-A(s))} \, \dfrac{g^{s-i}}{s-i}\, \varepsilon^{-s} \, ds \right | &\leq&
\int_{c_l}^c \left |\dfrac{{\rm adj}(I-A(t+\mathbbm{i} \tau_j))_{uv}}{\det(I-A(t+\mathbbm{i}\tau_j))} \right |   \, \dfrac{g^{t-i}}{| t+\mathbbm{i} \tau_j-i |} \, \varepsilon^{-t} \, dt
\end{eqnarray*}
${\rm adj}(I-A(s))_{uv}$ is of the form $\displaystyle \tilde{1}+\sum_\alpha p_\alpha^s - \sum_\beta q_\beta^s$ ($\tilde{1}$ indicates that $1$ might be present or absent) and therefore is bounded for $c_l\leq {\rm Re}(s) \leq c$, so that we can write
\begin{eqnarray*}
\left | \int_{L_{2,j}} \dfrac{{\rm adj}(I-A(s))_{uv}}{\det(I-A(s))} \, \dfrac{g^{s-i}}{s-i}\, \varepsilon^{-s} \, ds \right | &\leq&
C \, \int_{c_l}^c \dfrac{dt}{\tau_j} \quad \mbox{ ($C$ being a constant not depending on $j$)}
\end{eqnarray*}
which tends to zero for $\tau_j\to\infty$.

Similarly the integral on $L_{4,j}\to 0$ for $\tau_j\to\infty$.

Finally we consider the integral on $L_{3,j}$.
\begin{eqnarray*}
 \int_{L_{3,j}} \dfrac{{\rm adj}(I-A(s))_{uv}}{\det(I-A(s))} \, \dfrac{g^{s-i}}{s-i}\, \varepsilon^{-s} \, ds  &=&
 \int_{C_j} \dfrac{{\rm adj}(I-A(s))_{uv}}{\det(I-A(s))} \, \dfrac{g^{s-i}}{s-i}\, \varepsilon^{-s} \, ds
\end{eqnarray*}
where $C_j=c_l+\tau_j e^{\mathbbm{i}t}, \, \frac{\pi}{2}\leq t \leq \frac{3\pi}{2} $ (see Figure~\ref{semicircle}).  ${\rm adj}(I-A(s))_{uv}$ is of the form $\tilde{1}+\sum_\alpha p_\alpha^s - \sum_\beta q_\beta^s$. In any term, $p_\alpha , q_\beta$ there can appear at most $(N-1).$ power of the smallest weight $r_{{\rm min}}$ of the graph where $N$ is the number of the nodes of the graph. We can thus dominate $|{\rm adj}(I-A(s))_{uv}|$ by $C^\prime \, r_{{\rm min}}^{(N-1){\rm Re}(s)}$.
We had $|\det(I-A(s))|>\delta$ by the choice of $c_l$.

\begin{figure}[h]
\begin{center}
\begin{pspicture}(-2,-2.1)(4,2)
%\psunit 0.3cm
\psline [linewidth=0.5pt,arrowsize=4pt]{->}(-1,0)(4,0)
\psline[linewidth=0.5pt](2,-2)(2,2)
\psline[linewidth=0.5pt,arrowsize=4pt]{->}(2,2)(2,-1)
\psarc[linewidth=0.5pt](2,0){2}{90}{270}
\psarc[linewidth=0.5pt,arrowsize=4pt]{->}(2,0){2}{90}{150}
\uput[r](2,-1){$L_{3,j}$}
\uput[ur](2,0){$c_{l}$}
\uput[r](2,-2){$c_{l}-\mathbbm{i}\tau_j$}
\uput[r](2,2){$c_{l}+\mathbbm{i}\tau_j$}
\rput(0,1.5){$C_j$}
\qdisk(2,2){1pt}\qdisk(2,-2){1pt}\qdisk(2,0){1pt}
\end{pspicture}
\end{center}
\caption{ The semi-circle $C_j$ used to evaluate the integral on the segment $L_{3,j}$.}\label{semicircle}
\end{figure}
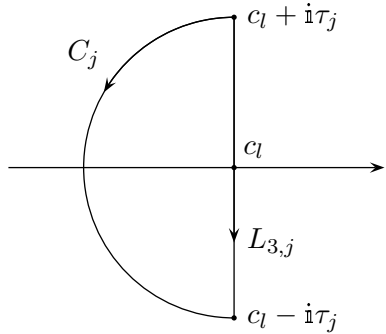

We can now write
\begin{eqnarray*}
\left| \int_{C_j} \dfrac{{\rm adj}(I-A(s))_{uv}}{\det(I-A(s))} \, \dfrac{g^{s-i}}{s-i}\, \varepsilon^{-s} \, ds \right|
&\leq&
\int_{C_j} \dfrac{C^\prime \, r_{{\rm min}}^{(N-1){\rm Re}(s)}}{\delta} \, \dfrac{g^{{\rm Re}(s)-i}}{|s-i|}\, \varepsilon^{-{\rm Re}(s)} \, | ds |\\
&\leq&
C^{\prime\prime} \, \int_{\frac{\pi}{2}}^{\frac{3\pi}{2}} \left(\dfrac{ r_{{\rm min}}^{N-1} \, g }{\varepsilon}\right)^{\tau_j \, \cos t}   \, \dfrac{\tau_j}{|c_l+\tau_j e^{\mathbbm{i}t} -i|} \, dt\\
&\leq&
C^{\prime\prime} \, \int_{\frac{\pi}{2}}^{\frac{3\pi}{2}} \left(\dfrac{ r_{{\rm min}}^{N-1} \, g }{\varepsilon}\right)^{\tau_j \, \cos t} \, dt
\end{eqnarray*}
since $\tau_j \leq |c_l+\tau_j e^{\mathbbm{i}t} -i|$. By the Jordan Lemma (which states that
\[\lim_{n\to\infty} \int_{\frac{\pi}{2}}^{ \frac{3\pi}{2}} a^{n
\cos t } dt=0,\] for any fixed $a>1$) this integral tends to zero if  \[ \dfrac{ r_{{\rm min}}^{N-1} \, g }{\varepsilon}>1.\]

So we come to the conclusion that our tube formula is valid pointwise for \linebreak \mbox{$\displaystyle \varepsilon < r_{{\rm min}}^{N-1} \, \min_{u\in V} \left\{  g_u \right\}$} for monophase generators. In the pluriphase case, one should take \mbox{$\displaystyle \varepsilon < r_{{\rm min}}^{N-1} \, \min_{u\in V} \left\{ g_{1,u} \right\}$}. Needless to say, these are only sufficient bounds.

\begin{remark}
There is an unfortunate misprint in \cite[p. 159, line 12]{DKOUJofG}, where $\varepsilon<g$ should be $\varepsilon<g_1$.
\end{remark}


\begin{thebibliography}{1}


\bibitem{Bateman} P.T. Bateman and H.G. Diamond, Analytic Number Theory, World Scientific, 2005.

\bibitem{DDKUFractals} B. Demir, A. Deniz, Ş. Koçak, and A. E. Üreyen, Tube formulas for graph-directed fractals, Fractals, \textbf{18}(3) (2010), 349--361.

\bibitem{DKOUIntelligencer} A. Deniz, Ş. Kocak, Y. Özdemir, A.E. Üreyen, Tube Formulas for Self-Similar and Graph-Directed Fractals, Math. Intelligencer, \textbf{35}(3) (2013), 36--49.

\bibitem{DKOUJofG} A. Deniz, Ş. Kocak, Y. Özdemir, A.E. Üreyen, Tube Volumes via Functional Equations, Journal of Geometry, \textbf{106} (2015), 153--162.
\bibitem{EdgarBook} G. Edgar, Measure, Topology and Fractal Geometry, 2$^{\rm{nd}}$ edn., Springer, New York, 2008.

\bibitem{HornMatrixBook} R.A. Horn, C.R. Johnson, Matrix Analysis (Second Edition), Cambridge Univ Press, 2013.

\bibitem{LaFraBook} M.L. Lapidus and M. van Frankenhuijsen, Fractal Geometry, Complex Dimensions and Zeta Functions: Geometry and
Spectra of Fractal Strings, (Springer Monographs in Mathematics,
Springer-Verlag, New York, 2006; 2$^{\rm{nd}}$ edn., 2012).

\bibitem{LaPeActa} M. L. Lapidus, E. P. J. Pearse, Tube formulas and complex dimensions of self-similar tilings, Acta Appl. Math., \textbf{112}(1) (2010), 91--136.

\bibitem{LPW} M. L. Lapidus, E. P. J. Pearse, and S. Winter, Pointwise tube formulas for fractal sprays and self-similar tilings with arbitrary generators,
 Advances in Mathematics, \textbf{227} (2011), 1349-1398.

\bibitem{LaPoCounter} M. L. Lapidus and C. Pomerance, Counterexamples to the modified Weyl-Berry conjecture on fractal drums,
Math. Proc. Cambridge Philos. Soc., \textbf{119} (1996), \mbox{167--178.}

\bibitem{LeVas} M. Levitin and D. Vassiliev, Spectral asymptotics, renewal theorem, and the Berry
conjecture for a class of fractals. Proc. Lond. Math. Soc., textbf{72}(3) (1996), 178--214.


\bibitem{PW} E.P.J. Pearse and S. Winter, Geometry of canonical
self-similar tilings, \emph{Rocky Mountain J.} 42(4) (2012), 1327--1357.


\bibitem{Titch} E.C. Titchmarsh, Introduction to the theory of Fourier integrals, Oxford at the Clerandon Press (1948)


\end{thebibliography}
\end{document}